\documentclass[10pt]{article}
\usepackage{amsmath}
\usepackage{amssymb}
\usepackage{amsthm}
\usepackage{graphicx} 
\usepackage{appendix}

\newtheorem{theorem}{Theorem}[section]
\newtheorem{lemma}[theorem]{Lemma}

\newtheorem{definition}[theorem]{Definition}

\newcommand{\Sup}[1]{\underset{#1}{sup\,}}

\newcommand{\R}{{\bf R}}

\newcommand{\lam}{\lambda}

\newcommand{\Ome}{\Omega}
\newcommand{\vecy}[1]{\boldsymbol{#1}}

\newcommand{\lap}{{\Delta}}

\newcommand{\bra}[1]{\left( #1 \right)}
\newcommand{\Quad}[1]{{Q \bra{#1}}}

\newcommand{\norm}[1]{\Vert #1 \Vert}

\newcommand{\beq}{\begin{equation}}

\newcommand{\mlap}[1]{\norm{\bra{-\lap}^{#1}f}}

\newcommand{\kik}[1]{C^{\infty}_{c}\bra{#1}}
\newcommand{\ip}[2]{\langle #1,#2\rangle}
\newcommand{\sbra}[1]{\small( #1 \small)}

\newcommand{\sspace}{\,\,}

\newcommand{\sobz}{W^{m,2}_0}
\newcommand{\dx}{d_x}

\newcommand{\g}{{\tau}}
\newcommand{\ggamma}{{m\bra{1-\epsilon}-\tfrac{N}{2}}}

\newcommand{\UUU}[2]{\tfrac{c}{#1}\sspace \dx^{{}^{2m\bra{1-\frac{N}{2m}-{#1}}}}\sspace  \g\bra{\tfrac{#2}{2}}^{1-{#1}}{e}^{-\mu {#2} }}
\newcommand{\UU}[2]{u_{#1}\bra{#2,x}}
\newcommand{\DD}[2]{\delta_{#1}\bra{#2,x}}
\newcommand{\modulus}[1]{\arrowvert#1\arrowvert}

\begin{document}
\title{\bf Pointwise Lower bounds on the Heat Kernels of Uniformally Elliptic Operators in Bounded Regions}
\author{Narinder S Claire}
\date{}
\maketitle
\begin{abstract}
We obtain pointwise lower bounds for heat kernels of higher order differential operators with Dirichlet boundary conditions on 
bounded domains in $\R^N$. The bounds  exhibit explicitly the nature of the spatial decay of the heat kernel close to the boundary. We make no smoothness 
assumptions on our operator coefficients which we assume only to be bounded and measurable.\\ \\
{\bf AMS Subject Classification : 35K25}  
\end{abstract} 

\section{Introduction}
Pointwise lower bounds on the heat kernels for higher order elliptic operators were first obtained
by Davies \cite{Da4}. 
In this paper we address the question of pointwise lower bounds on heat kernels generated by uniformly elliptic differential operators 
with Dirichlet boundary conditions on bounded regions of $\R^N$.

It is helpful to give an indicative though a non-rigorous formulation of the family of higher order operators that we focus on 
in this paper. The operator is defined more completely through it's quadratic form.
Given a bounded domain $\Ome$ in $\R^N$ we express the operator of order $2m>N$ as : 
\begin{equation}
 Hf\bra{\vecy{x}}\quad :=\quad \sum\limits_{\substack{|\vecy{\alpha}|\leq m \\ |\vecy{\beta}|\leq m}}
 \bra{-1}^{|\vecy{\alpha}|}
D^{\vecy{\alpha}}\bra{a_{\vecy{\alpha},\vecy{\beta}}\bra{\vecy{x}}D^{\vecy{\beta}}f\bra{\vecy{x}}}
\end{equation}
where $a_{\vecy{\alpha},\vecy{\beta}}$ are complex bounded measurable functions.\\ 
The associated quadratic form $Q$ 
\begin{equation}
Q\bra{f}\quad :=\quad \sum\limits_{\substack{|\vecy{\alpha}|\leq m \\ |\vecy{\beta}|\leq m}}
\int_{\substack{\Ome}}a_{\vecy{\alpha},\vecy{\beta}}\bra{\vecy{x}}D^{\vecy{\beta}}f\bra{\vecy{x}}\overline{D^{\vecy{\alpha}}f\bra{\vecy{x}}}
\end{equation}

defined with domain equal to the Sobolev space 
$\sobz \bra{\Ome}$ will be assumed to satisfy the ellipticity condition with a strictly positive constant $c$ 
\begin{equation}\label{H1}
c^{-1}\mlap{\frac{m}{2}}^2_2\leq\Quad{f}\leq c\mlap{\frac{m}{2}}^2_2
\end{equation}
We define the the spectral gap
$$ \mu=\inf\limits_{f\in \kik{\Ome}}\frac{\Quad{f}}{\norm{f}_2^2}$$
Since we have made the assumption that $\tfrac{N}{2m}<1$, it will be informative to track the dependency on this constraint by defining the 
quantity $0<\epsilon\leq1-\tfrac{N}{2m}$ and 
$$\gamma:=\ggamma$$ \par
For a given point $x$ in $\Ome$ we define it's distance from the boundary $\partial \Ome$ 
$$\dx:= d\bra{\vecy{x},\partial\Ome} = \inf\limits_{\vecy{y}\in \partial\Ome}\modulus{\vecy{x} - \vecy{y}}$$
We define the function $\g\bra{t}$ such that 
\begin{equation}\label{e:gt}
\g\bra{t} := \begin{cases} \mu  e^{-2\mu t} &  t>\frac{1}{\mu }\\
\frac{1}{t}e^{-\mu  t-1} &  t \leq \frac{1}{\mu }\\
\end{cases}
\end{equation}
Our main result is the following theorem :
\begin{theorem}
If the heat kernels generated by the differential operator $H$ satisfies the inequalities
\begin{eqnarray*}
k\sbra{t,x,x}\leq c\sspace{\small(1-\tfrac{N+2\gamma}{2m}\small)}^{-1}\sspace t^{-\frac{N+2\gamma}{2m}}\sspace \dx^{2\gamma}\sspace& {\text{when }} t<\frac{2}{\mu}\\
\\
k\sbra{t,x,x}\leq c\sspace{\small(1-\tfrac{N+2\gamma}{2m}\small)}^{-1}\sspace e^{-\mu t}\sspace\dx^{2\gamma}\sspace &  {\text{when }} t\geq\frac{2}{\mu}
\end{eqnarray*}

\end{theorem}

\par
We define the function $\UU{\epsilon}{t}$ such that
$$
\UU{\epsilon}{t} := \UUU{\epsilon}{t}
$$
with $c>0$ such that $$k\sbra{t,x,x}\leq \UU{\epsilon}{t}$$
and define the function $$\DD{\epsilon}{t} := \frac{k\sbra{t,x,x}}{\UU{\epsilon}{t}}$$

We summarise some of the key lemmas of Davies'\cite{Da4} theory.
\begin{lemma}[Davies \cite{Da4}] \label{dav}
Given  $0<\alpha,s<1$ and $p$ defined by 
\begin{equation}
p+\bra{1-p}\alpha s=s
\end{equation}
$$k\bra{ts,x,x}<\UU{\epsilon}{\alpha ts}^{1-p}
\UU{\epsilon}{t}^p\bra{\frac{k\bra{t,x,x}}{\UU{\epsilon}{t}}}^p$$
\end{lemma}

\begin{definition}
If $\omega_x$ is the distribution such that 
$\ip{f}{\omega_x}=f\bra{x}$ then  we define $G_t\bra{x,x}$
$$G_t\bra{x,x}:= \langle \bra{tH+1}^{-1}\omega_x,\omega_x \rangle = \langle \int\limits_0^{\infty}e^{-\bra{tH+1}s} ds 
\ \omega_x,\omega_x\rangle$$
\end{definition}
\begin{lemma}[Davies \cite{Da4}] \label{dav2}
 \begin{equation}
\frac{G_t\bra{x,x}}{\UU{\epsilon}{t}}<\int\limits_0^1 \frac{\UU{\epsilon}{\alpha ts}}{\UU{\epsilon}{t}}\,\delta_\epsilon \bra {t,x}^p\,ds+e^{-1}\delta_\epsilon\bra {t,x}
\end{equation}
\end{lemma}

\begin{lemma}
$$
G_t\bra{x,x}  = \Sup{g \in Dom{Q}}\{\frac{|g\bra{x}|^2}{t\Quad{g}+\norm{g}^2_2} \, :g \neq 0\}\\
$$
\end{lemma}
\begin{proof}
\begin{eqnarray*}
G_t\bra{x,x} & = & \ip{\bra{tH+1}^{-1}\omega_x}{\omega_x}\\
        & = & \ip {\bra{tH+1}^{-\frac{1}{2}}\omega_x}{\bra{tH+1}^{-\frac{1}{2}}\omega_x}\\
        & = & \Sup{f \in L^2\bra{\Ome}}\{|\ip {\bra{tH+1}^{-\frac{1}{2}}\omega_x}{f}|^2\, : \norm{f}_2=1\}\\
        & = & \Sup{f \in L^2\bra{\Ome}}\{|\bra{tH+1}^{-\frac{1}{2}}f\bra{x}|^2 : \norm{f}_2=1\}\\
        & = & \Sup{g \in Dom{Q}}\{\frac{|g\bra{x}|^2}{\norm{\bra{tH+1}^{\frac{1}{2}}g}^2} \, :g \neq 0\}\\
        & = & \Sup{g \in Dom{Q}}\{\frac{|g\bra{x}|^2}{t\Quad{g}+\norm{g}^2_2} \, :g \neq 0\}\\
\end{eqnarray*}
\end{proof}

\section{Estimation with Test Functions}
Let  $\psi$ be a function in $\kik{\Ome}$ defined as :
$$ \psi \bra{u} := \begin{cases}
1 & u=0 \\
0  & |u| \geq 1\\
\end{cases}
$$
and then define the test function $g_{\bra{x,r}}$ as 
$$g_{\bra{x,r}}\bra{y}=\psi\bra{\frac{y-x}{r}}$$
It is then clear to see that $g_{\bra{x,r}} \in Dom\bra{Q}$ iff $ x \in \Ome$ and $r \leq \dx$.\\
By direct integration we can see that we have the two estimates
$$
\norm{g_{\bra{x,r}}}_2^2  \quad \leq \quad \int\limits_{B\bra{x,r}} d^N\vecy{y} \quad \leq  \quad cr^N
$$
$$
\Quad{g_{\bra{x,r}}}  \quad  \leq  \quad   \int\limits_{B\bra{x,r}}  |\lap^{\frac{m}{2}}g_{\bra{x,r}}\bra{y}|^2 d^N\vecy{y}  \quad \leq  \quad  c r^{N-2m}
$$ 
for some constant $c>0$. As a consequence we have the estimate 
\begin{equation}\label{gestimate}
G_{t}\bra{x,x} \geq\frac{c}{tr^{N-2m}+r^N}
\end{equation}
for some constant $c>0$.

\begin{lemma}\label{gineq} There is a constant $c>0$ such that \\
when $t \leq \dx^{2m}$ then 
\begin{equation}
G_{t}\bra{x,x} \geq c\,t^{-\frac{N}{2m}}
\end{equation}
when $t \geq\dx^{2m}$ then 
 \begin{equation}\label{e:nike}
G_{t}\bra{x,x} \geq c\, t^{-1}\dx^{2m-N}
\end{equation}
\end{lemma}
\begin{proof}
When  $t \leq \dx^{2m}$
we substitute  $r=t^{\frac{1}{2m}}$ into inequality \eqref{gestimate}.\\
When $t \geq \dx^{2m}$
we substitute  $r=\dx$ into inequality \eqref{gestimate}.
\end{proof}

\section{Short Time asymptotics}
In this section we assume that  $t < \frac{2}{\mu}$.
\begin{lemma}
 When $t\leq \dx^{2m}$ then  $k\bra{t,x,x}> c\,t^{-\frac{N}{2m}}$ for some constant $c>0$
\end{lemma}
\begin{proof}
By lemma \ref{gineq} we have 
$$\frac{G_{t}\bra{x,x}}{\UU{\epsilon}{t}}\geq
c\,\epsilon \bra{\frac{t}{\dx^{2m}}}^{1-\frac{N}{2m}-\epsilon} e^{\frac{\mu t}{2}\bra{1-\epsilon}}$$

We begin our estimate with
$$ \alpha\int\limits_0^\alpha v\bra{\theta,Ts,x} \delta_\theta^{\frac{s}{\alpha}} ds  <
\alpha\int\limits_0^\infty v\bra{\theta,Ts,x} \delta_\theta^s ds$$
recalling that $\alpha<1$ and $\delta<1$.

Now substituting $v$ we have
$$ \alpha\int\limits_0^\infty v\bra{\theta,Ts,x} \delta_\theta^s \ ds =a\alpha
c_\theta^2\frac{\dx^{2m\theta-N}}{T^{\theta}}
\int\limits_0^\infty s^{-\theta}e^{-s Ts\bra{1-\frac{\theta}{2}}}e^{s\ln \delta_\theta}\ ds$$
and evaluating this we have
$$ \alpha\int\limits_0^\infty v\bra{\theta,Ts,x} \delta_\theta^s \ ds = \alpha
c_\theta^2\frac{\dx^{2m\theta-N}}{T^{\theta}}
\big[\ln\bra{\frac{e^{s T\bra{1-\frac{\theta}{2}}}}{\delta_\theta}}\big]^{\theta-1}
\Gamma\bra{1-\frac{\theta}{2}}$$ where $a$ is some constant,
hence we conclude

$$\frac{\alpha}{v\bra{\theta,T,x}}\int\limits_0^\alpha u\bra{\theta,Ts,x}
 \delta_\theta^{\frac{s}{\alpha}} ds  <e^{s T\bra{1-\frac{\theta}{2}}}
\big[\ln\bra{\frac{e^{s T\bra{1-\frac{\theta}{2}}}}{\delta_\theta}}\big]^{\theta-1}
\Gamma\bra{1-\frac{\theta}{2}}$$
and with these estimates we conclude from the Green-Heat inequality
$$a\frac{e^\theta}{c^2_\theta}\bra{\frac{\dx^{2m}}{T}}^{\frac{N}{2m}-\theta}e^{s T\bra{1-\frac{\theta}{2}}}
<e^{s T\bra{1-\frac{\theta}{2}}}\big[\ln\bra{\frac{e^{s T\bra{1-\frac{\theta}{2}}}}{\delta_\theta}}\big]^{\theta-1}
\Gamma\bra{1-\frac{\theta}{2}}+e^{-1}\delta_\theta$$
then dividing by $e^{s T\bra{1-\frac{\theta}{2}}}$ we have
$$\frac{e^\theta}{c^2_\theta}\bra{\frac{\dx^{2m}}{T}}^{\frac{N}{2m}-\theta}
<\big[\ln\bra{\frac{1}{\kappa}}\big]^{\theta-1}
\Gamma\bra{1-\frac{\theta}{2}}+e^{-1}\kappa$$
where $\kappa=\frac{\delta}{e^{s T\bra{2-\theta}}}$
We now set $\theta=\frac{N}{2m}$ and by applying lemma \ref{lem:gamma} we see that for some constant c
$$ c<\big[\ln\bra{\frac{1}{\kappa}}\big]^{\frac{N}{2m}-1}$$and
$$c<\kappa$$
where $c$ is some positive constant, this is
$$c<e^{-s T\bra{1-\frac{N}{4m}}}\delta_{\frac{N}{2m}}$$
and hence
$$k\bra{T,x,x}> ce^{s T \bra{1-\frac{N}{4m}}}v\bra{\tfrac{N}{2m},T,x}$$
Making the substitution $u\bra{\tfrac{N}{2m},T,x}$
$$k\bra{T,x,x}> \frac{c}{T^{\frac{N}{2m}}}$$
 
\end{proof}

\subsection{When $T\geq \dx^{2m}$}
The only difference to our analysis from the case $T \leq \dx^{2m}$ is the lower bound on the
Greens function, here it is $$G_T\bra{x}>c\frac{\dx^{2m}}{T}$$
and this leads to the different inequality
\begin{eqnarray*}
\frac{G_{T}\bra{x}}{u\bra{\theta,T,x}}&>&
\frac{e^\theta}{c^2_\theta}\frac{\dx^{2m-N}}{T}\frac{T^\theta}{\dx^{2m\theta-N}}
e^{s T\bra{1-\frac{\theta}{2}}}\\
&=& \frac{e^\theta}{c^2_\theta}
\bra{\frac{\dx^{2m}}{T}}^{1-\theta}e^{s T\bra{1-\frac{\theta}{2}}}
\end{eqnarray*}
With this the Green-Heat Inequality becomes
$$\frac{e^\theta}{c^2_\theta}
\bra{\frac{\dx^{2m}}{T}}^{1-\theta}e^{s T\bra{1-\frac{\theta}{2}}}<
e^{s T\bra{1-\frac{\theta}{2}}}\big[\ln\bra{\frac{e^{s T\bra{1-\frac{\theta}{2}}}}{\delta_\theta}}\big]^{\theta-1}
\Gamma\bra{1-\frac{\theta}{2}}+e^{-1}\delta_\theta$$
and as before we let $\kappa=\frac{\delta_\theta}{e^{s T\bra{1-\frac{\theta}{2}}}}$ and divide by
$e^{s T\bra{1-\frac{\theta}{2}}}$, to get
$$\frac{a}{c^2_\theta}
\bra{\frac{\dx^{2m}}{T}}^{1-\theta}<\big[\ln\bra{\frac{1}{\kappa}}\big]^{\theta-1}
\Gamma\bra{1-\frac{\theta}{2}}+e^{-1}\kappa$$
We continue by applying lemma \ref{lem:gamma} to conclude
$$\frac{a}{c_\theta^2}\bra{\frac{\dx^{2m}}{T}}^{1-\theta}<
\big[\ln\bra{\frac{1}{\kappa}}\big]^{\theta-1}$$
for some constant $a$
We next obtain from this
$$a c_\theta^{\frac{2}{1-\theta}}\bra{\frac{\dx^{2m}}{T}}^{\frac{1-\theta}{\theta-1}}
=ac_\theta^{\frac{2}{1-\theta}}\bra{\frac{\dx^{2m}}{T}}^{-1}>\ln\bra{\frac{1}{\kappa}}$$
and inverting the logarithm and substituting for $\kappa$
$$ac'_\theta e^{s T\bra{2-\theta}}e^{-\frac{T}{\dx^{2m}}}\leq \delta_\theta$$
and so
$$k\bra{T,x,x}>a c_\theta c'_\theta \frac{\dx^{2m\theta-N}}{T}e^{-\frac{T}{\dx^{2m}}}$$

\section{Long Time Asymptotics}
We now consider the situation where $T>\frac{1}{s}$
As before we begin with the inequalities
$$G_T\bra{x}>c\frac{\dx^{2m-N}}{T}$$
and the upper bound on the heat kernel for $T>\frac{1}{s}$ is
$$u\bra{\theta,T,x}=c^2_{\theta}s^{-\theta}\dx^{2m-N}e^{-s T}$$
therefore
\begin{eqnarray*}
\frac{G_T\bra{x}}{u\bra{\theta,T,x}}&>&c^{-2}_{\theta}s^{-\theta}\bra{\frac{\dx^{2m-N}}
{T}}\dx^{2m}\frac{e^{s T}}{\dx^{2m-N}}\\
&=&c^{-2}_{\theta}s^{-\theta}\frac{e^{s T}}{T}
\end{eqnarray*}
From section $6.1$ we deduce from the Green-Heat inequality 
\begin{equation}
\frac{G_T\bra{x,x}}{u\bra{\theta,T,x}}<\int\limits_0^{\infty}
 \frac{u\bra{\theta,\alpha Ts,x}}{u\bra{T,x}}\bra{\delta_\theta \bra {T,x}}^p+e^{-1}\delta_\theta
\end{equation}
It is important to note at this point that $u\bra{\theta,\alpha Ts,x}$ is formulated
 differently
for $s>\frac{1}{\alpha s T}$ and $s<\frac{1}{\alpha s T}$ and so we use the estimate
$$\frac{G_T\bra{x,x}}{u\bra{\theta,T,x}}<\int\limits_0^{\infty}
 \frac{v\bra{\theta,\alpha Ts,x}}{u\bra{T,x}}\bra{\delta_\theta \bra {T,x}}^p+e^{-1}\delta_\theta$$

Giving us the Green-Heat inequality
$$c^{-2}_{\theta}s^{-\theta}\frac{e^{s T}}{T}
\leq \frac{e^{s T}}{\dx^{2m}}\bra{\frac{\dx^{2m}}{T}}^\theta
\big[\ln\bra{\frac{e^{s T\bra{1-\frac{\theta}{2}}}}{\delta_\theta}}\big]^{\theta-1}
\Gamma\bra{1-\frac{\theta}{2}}+e^{-1}\delta_\theta$$
we then multiply by $e^{s T\bra{\frac{\theta}{2}-1}}\dx^{2m}$

\begin{eqnarray*}
c^{-2}_{\theta}s^{-\theta}e^{\frac{s T\theta}{2}}\frac{\dx^{2m}}{T}
\leq e^{\frac{s T\theta}{2}}\bra{\frac{\dx^{2m}}{T}}^\theta
\big[\ln\bra{\frac{e^{s T\bra{1-\frac{\theta}{2}}}}{\delta_\theta}}\big]^{\theta-1}
\Gamma\bra{1-\frac{\theta}{2}}\\+e^{-1}\dx^{2m}e^{s T\bra{\frac{\theta}{2}-1}}\delta_\theta
\end{eqnarray*}
By observing that $\dx$ is bounded above and that 
$$\frac{\dx^{2m}}{T}<1$$ we get by setting 
$\kappa=\frac{\delta_\theta}{e^{s T\bra{1-\frac{\theta}{2}}}}$ and applying lemma \ref{lem:gamma}

$$ac^{-2}_{\theta}s^{-\theta}\bra{\frac{\dx^{2m}}{T}}\leq c
\big[\ln\frac{1}{\kappa}\big]^{\theta-1}$$
and then 
$$ac^{\frac{2}{1-\theta}}_{\theta}\bra{\frac{\dx^{2m}}{T}}^{\frac{1}{\theta-1}}\geq c
\big[\ln\frac{1}{\kappa}\big]$$

and solving we get our lower bound
$$k\bra{T,x,x}>ac_\theta'u\bra{\theta,T,x}
\exp\bra{-\bra{\frac{T}{\dx^{2m}}}^{\frac{1}{1-\theta}}}$$
hence
$$k\bra{T,x,x}>ac_\theta^2 c_\theta'\dx^{2m-N}e^{-s T}\exp\bra{-\bra{\frac{T}{\dx^{2m}}}^{\frac{1}{1-\theta}}}$$

\par
The situation for higher order operators is very different than for the case of second order
differential operators. The heat kernel may be written as
$$k\bra{t,x,y}=\sum\limits_{n=0}^\infty e^{-\lam_n t}\psi_n\bra{x}\psi_n\bra{y}$$
where $\psi_n$ are the eigenfunctions and $\lam_n$ are the eigenvalues given in
increasing order. In the case of second order operators the ground state $\psi_0$
is positive.
We therefore have the asymptotic behaviour
$$ k\bra{t,x,x}\sim e^{-\lam_0 t}\psi\bra{x}^2$$
for large time. However since the ground state for higher order operators
is not necessarily positive this type of behaviour is not necessarily true.

\section*{Acknowledgements}
This research was funded by an EPSRC Ph.D grant 95-98 and my mother. I would like to thank Brian Davies
for giving me this problem and his encouragement since. I am grateful to all the research students
 in the Mathematics Dept. Kings College, London 1995-2000 for all their help and support. I thank Owen Nicholas and Mark Owen
 for their invaluable discussions. I am indebted 
 to Anita for all her support.

\vskip 0.3in
\end{document}